\documentclass[reqno]{amsart}
\usepackage[utf8]{inputenc} 

\usepackage{graphicx} 
\usepackage{parskip}

\usepackage[x11names]{xcolor}
\usepackage[colorlinks,
            breaklinks=true,
            allcolors=SpringGreen4,
            urlcolor=blue]{hyperref}

\usepackage{amsfonts,amsmath,amssymb,amsthm}
\usepackage{enumitem}

\usepackage[round]{natbib}

\calclayout

\usepackage[nameinlink, capitalise]{cleveref}

\newcommand{\MMsays}[1]{\bgroup\color{blue}Mattes:~#1\egroup}
\newcommand{\N}{\mathbb{N}} 
\newcommand{\R}{\mathbb{R}}

\newcommand{\sigalg}{\mathcal{F}}

\newcommand{\E}{\mathbb{E}}
\renewcommand{\P}{\mathbb{P}}

\newcommand{\X}{\mathcal{X}}
\newcommand{\Y}{\mathcal{Y}}

\DeclareMathOperator{\tr}{tr}

\DeclareMathOperator{\Cov}{\mathbb{C}}

\DeclareMathOperator{\mspan}{span}

\DeclareMathOperator{\idop}{Id}

\newcommand{\interior}[1]{%
  {\kern0pt#1}^{\mathrm{o}}%
}

\newcommand{\norm}[1]{\lVert #1 \rVert}
\newcommand{\innerprod}[2]{\langle #1,\, #2 \rangle}

\newtheorem{theorem}{Theorem}[section]
\newtheorem{corollary}[theorem]{Corollary}
\newtheorem{lemma}[theorem]{Lemma}

\newtheorem{proposition}[theorem]{Proposition}
\newtheorem{definition}[theorem]{Definition}

\theoremstyle{definition}

\theoremstyle{remark}
\newtheorem{remark}[theorem]{Remark}
\newtheorem{example}[theorem]{Example}

\title[Subgaussian concentration in Hilbert spaces]{
    On the concentration of subgaussian vectors\\
    and positive quadratic forms in Hilbert spaces
    }

\usepackage{amsaddr}

\author{\vspace{-.5cm} Mattes Mollenhauer \\Claudia Schillings}
\address{\vspace{-.5cm}Freie Universität Berlin\\
Department of Mathematics and Computer Science \\
\href{mailto:mattes.mollenhauer@fu-berlin.de}{mattes.mollenhauer@fu-berlin.de}\\
\href{mailto:c.schillings@fu-berlin.de}{c.schillings@fu-berlin.de}}

\begin{document}
\bibliographystyle{abbrvnat}
\setcitestyle{authoryear, aysep={ }, comma}

\maketitle

\begin{abstract}
    \vspace{-1cm}
    In these notes, we investigate the tail behaviour of the norm of subgaussian vectors
    in a Hilbert space.
    The subgaussian variance proxy is given as a trace class operator,
    allowing for a precise control of the moments along each dimension
    of the space.
    This leads to useful extensions and analogues of known
    Hoeffding-type inequalities and deviation bounds
    for positive random quadratic forms.
    We give a straightforward application in terms of a variance bound
    for the regularisation of statistical inverse problems.
    \textbf{Subject.} Primary: 60E15, 66G15, 60G50, Secondary: 46N30
\end{abstract}

\section{Introduction}

The concentration of random vectors and their sums in infinite-dimensional
spaces is a central topic in modern probability.
Tail estimates for these vectors are generally derived by
controlling either the moments of their norms
or the suprema of their weak moments over all dual evaluations,
see 
\citet{LedouxTalagrand1991},
\citet{Pinelis1994}, \citet{Yurinsky1995},
\citet{Bousquet2002} and \citet{MaurerPontil2021}.
Arguably, these assumptions are suboptimal
in cases where the weak moments are relatively small
along ``most'' (e.g.\ in all but finitely many) dimensions of 
the space. This may lead to imprecise bounds---especially
when considering random
vectors under linear transformations which are ``compatible''
with the structure of the weak moments.
Consequently, the purpose of these notes is to investigate
the concentration of unbounded random vectors
with sufficiently fast decaying weak subgaussian variance proxies
measured in terms
of a trace class operator \citep{Fukuda1990,Antonini1997}.
We give a bound on the moment generating function
which can be interpreted as a quantitative version of the well-known
\emph{Fernique theorem}.
Our notes complement and sharpen a bound obtained for a special case of
recent results by \citet{ChenYang2021}, who
prove versions of the more general so-called
\emph{Hanson-Wright} inequality under similar assumptions.
Specifically, we discuss
a generalisation of the following classical result 
for finite-dimensional Gaussian vectors.

\begin{proposition}[\citealt{LaurentMassart2000}, Lemma 1%
    \footnote{The original result is formulated for the case that $\Cov[X]$
    is diagonal, the general case follows by the rotational invariance of all
    terms involved in the bound.}]
    \label{lem:laurent}
    Let $X$ be a centered Gaussian vector
    in the Euclidean space $\R^d$ with covariance matrix
    $\Cov[X]$.
    Then we have
    \begin{equation*}
        \log \E \left[ 
            e^{ \lambda \norm{ X }^2 }
        \right] 
        \leq 
        \tr( \Cov[X] )  +
        \frac{\lambda^2 \norm{\Cov[X]}^2_F}%
        {1 - 2 \lambda \norm{\Cov[X]}},
        \quad 0 \leq \lambda < 1/2\norm{\Cov[X]}.
    \end{equation*}
    In particular, this implies the tail bound
    \begin{equation*}
        \P\left[ \norm{ X }^2
            > 
            \tr( \Cov[X] )  + 2 \sqrt{t} \norm{\Cov[X] }_F 
            + 2 t \norm{ \Cov[X] }
        \right] 
        \leq e^{-t},
        \quad t \geq 0.
    \end{equation*}
\end{proposition}

These notes are organised as follows:
we introduce subgaussianity in Hilbert spaces in
\Cref{sec:main} and present the resulting concentration
of the norm of random vectors and corresponding positive quadratic forms.
In \Cref{sec:sums}, we discuss Hoeffding-type bounds for
sums of subgaussian vectors resulting from these considerations.
We extend our results to quadratic forms induced by random operators
in \Cref{sec:random}.
Related work is discussed in \Cref{sec:related_work}
and applications
of our theory in the context of regularised statistical inverse problems
are given in \Cref{sec:inverse}.
    
\section{Random vectors in Hilbert spaces}

We investigate the tail behaviour of subgaussian vectors in Hilbert spaces.
\label{sec:main}
\subsection{Setting and notation}
We consider a Hilbert space $\X$ which we assume to be real
and separable for simplicity
(the nonseparable case requires the assumption that
all considered random variables are
almost surely separably-valued, i.e. their distribution
is given by a \emph{Radon measure} on $\X$).
Random vectors taking values in $\X$ are interpreted
as measurable functions from a probability space
$(\Omega, \sigalg, \P)$ to the Borel space associated with $\X$ in the
\emph{Bochner sense} (see e.g.\ \citealp{VakhaniaEtAl1987} for
the mathematical background).
Every random vector $X$ in $\X$ satisfying
the Bochner square integrability condition
$\norm{X}^2_{L^2(\P;\X)} := \E[ \norm{ X}^2_\X ] < \infty$
admits the positive self-adjoint \emph{covariance operator} 
$\Cov[X] : \X \to \X$ defined by
$\innerprod{u}{\Cov[X] v}_\X = \E[ \innerprod{ X}{u}_\X \innerprod{X}{v}_\X]$
for all $u, v \in \X$. We have $\E[ \norm{ X}^2_\X ] = \tr( \Cov[X] ) $.
In particular, the operator $\Cov[X]$ is trace class.
Given another real separable Hilbert space $\Y$, 
we write $L(\X, \Y)$ and $S_1(\X, \Y)$ for the
Banach spaces of bounded operators and trace class operators 
from $\X$ to $\Y$, respectively. The Hilbert space
of Hilbert--Schmidt operators will be written as $S_2(\X, \Y)$.
If $\X = \Y$, we abbreviate $L(\X, \Y) = L(\X)$ and similarly
for the spaces $S_1(\X)$ and $S_2(\X)$, which form 
two-sided ideals in $L(\X)$.
We assume the reader is familiar with 
spectral and singular decompositions of Hilbert space operators
and their connections to the spaces given above
(see e.g.\ \citealp{Weidmann1980}).
For two self-adjoint operators $A,B \in L(\X)$ we write
$A \preceq B$ if $B$ dominates $A$ in the Loewner partial order,
meaning that $B-A$ is positive, i.e.\ we have 
$\innerprod{u}{(B-A)u} \geq 0 $ for all
$u \in \X$.

\subsection{Subgaussianity in Hilbert spaces}

An integrable centered real-valued random variable $\xi$ is called
\emph{$\sigma^2$-subgaussian}, if there exists
$\sigma^2>0$ such that we have
$\log \E[ e^{\lambda \xi}] \leq \frac{ \lambda^2 \sigma^2 }{2}$
for all $\lambda \in \R$. 
We refer the reader to 
\citet{Vershynin2018} for a variety of equivalent definitions
of subgaussianity
allowing for the case that $\xi$ is uncentered, which we will
not consider here in explicit form. The setting considered here
can be translated to the uncentered case by considering
$X- \E[X]$.

An integrable centered random vector $X$ in $\X$ is called
$\sigma^2$\emph{-weakly subgaussian},
if there exists some $\sigma^2>0$ such that
\begin{equation*}
    \log\E[ e^{\innerprod{u}{X}_\X} ] 
    \leq \frac{\sigma^2 \norm{u}^2_\X}{2}
    \quad \text{for all } u \in \X.
\end{equation*}
We now introduce
$R$\emph{-subgaussianity} in Hilbert spaces \citep{Antonini1997}.

%

\begin{definition}[$R$-subgaussianity]
    Let $X$ be an integrable centered random vector taking values in $\X$
    and $R: \X \to \X$ be a positive self-adjoint trace class operator.
    Then $X$ is called $R$\emph{-subgaussian}, if we have
    \begin{equation}
    \label{eq:subgaussian}
    \log \E[ e^{\innerprod{u}{X}_\X}] \leq \frac{ \innerprod{u}{Ru}_\X}{2}
    \quad \text{for all } u \in \X.
    \end{equation}
\end{definition}

It is clear that this definition is equivalent
to the existence of a random vector
$\gamma \sim \mathcal{N}(0, R)$ in $\X$ such that
for all $u \in \X$, we have
\begin{equation*}
    \E[ e^{\innerprod{u}{X}_\X}] 
    \leq 
    \E[ e^{\innerprod{u}{\gamma}_\X}].
\end{equation*}
This property is sometimes 
also called \emph{$\gamma$-subgaussianity}
in the context of Banach spaces \citep{Fukuda1990}.
While weak subgaussianity is equivalent to 
$\gamma$-subgaussianity in finite-dimensions,
weak subgaussianity does not imply $\gamma$-subgaussianity
in infinite dimensions.
We refer the reader to the recent exposition
by \citet{GiorgobianiEtAl2020} for more details
and connections between weak subgaussianity,
$\gamma$-subgaussianity and alternative concepts of subgaussianity in
infinite dimensions.

The operator $R$ has an advantage over the weak subgaussian variance 
proxy:
it clearly allows a more accurate control
of the moments of $X$ across its individual dimensions.
Many properties of $R$-subgaussians can be obtained straightforwardly
by applying the classical theory of real-valued subgaussian random
variables 
to the one-dimensional projections $\innerprod{u}{X}_\X$ for all $u \in \X$.
In particular, any centered Gaussian vector
$X$ in $\X$
with covariance operator $\Cov[X]$ is $\Cov[X]$-subgaussian.
Also note that if a random vector $X$ is $R$-subgaussian, then it is
$\norm{R}$-weakly subgaussian.

\subsection{Concentration of $R$-subgaussian vectors}
\label{sec:concentration}

The following gives an upper bound for the exponential
integrability of $R$-subgaussian vectors given by version
of \emph{Fernique's Theorem} \citep[Theorem 3.4]{Fukuda1990},
leading to a concentration bound which
directly generalises \Cref{lem:laurent}.

\begin{proposition}[Concentration of squared norm]
\label{prop:main}
Let $\X$ be a separable Hilbert space and
$X$ be $R$-subgaussian in $\X$.
Then we have the cumulant-generating function bound
\begin{equation}
    \label{eq:main_cgf}
    \log \E[ e^{\lambda \norm{X}^2_\X } ]
    \leq 
    \lambda \tr( R ) + \frac{ \lambda^2\norm{R}_{S_2(\X)}^2 }{ 1 - 2 \lambda \norm{ R }},
    \quad
    0 \leq \lambda < 1/2\norm{R}.
\end{equation}
In particular, this implies the tail bound
\begin{equation}
    \label{eq:main}
    \P\left[ 
    \norm{X}_{\X}^2 
    >
    \tr(R) +
    2 \sqrt{t} \norm{R}_{S_2(\X)} 
    + 2 t \norm{R} 
    \right] 
\leq e^{-t}, \quad t \geq 0.
\end{equation}
\end{proposition}

A short and elementary proof is provided \Cref{sec:proof}. 
It combines a standard Gaussian majorisation with a
monotone convergence argument
(see \citealp{Yurinsky1995}, Lemma 2.4.1), allowing to derive
the statement by simply applying \Cref{lem:laurent}.

\begin{remark}[Sub-gamma]
    \label{rem:sub_gamma}
    The cumulant-generating function bound \eqref{eq:main_cgf} shows that
    the real-valued random variable $\norm{X}_\X^2 - \tr(R)$ 
    is of \emph{sub-gamma} type,
    see \citet[Section 2.4]{BLM2013} and 
    the discussion of related work in \Cref{sec:related_work}.
\end{remark}

The bound for the squared norm of $X$ given in \eqref{eq:main} 
allows to straightforwardly derive 
useful $R$-subgaussian analogues for the norm of $X$ 
by taking the square root of
$\tr(R) + 2 \sqrt{t} \norm{R}_{S_2(\X)} + 2t \norm{R}$.
A simplified (but less precise estimate) can be derived 
by noting that we have
\begin{equation}
\label{eq:binomial}
\tr(R) + 2 \sqrt{t} \norm{R}_{S_2(\X)} + 2t \norm{R}
\leq
\left( \sqrt{ \tr(R) } + \sqrt{ 2 t \norm{R}} \right)^2
\quad 
\text{for all } t >0,
\end{equation}
where we use the binomial formula
and the fact $ \norm{R}_{S_2(\X)}^2  \leq \norm{R} \tr(R) $
(resulting from the definition of
the trace, the Hilbert--Schmidt norm and the operator norm
as $\ell^1$, $\ell^2$ and $\ell^\infty$ norms of the eigenvalues of $R$,
respectively).
Inserting \eqref{eq:binomial} into \eqref{eq:main},
taking the square root and expressing the tail bound
as a deviation bound shows that with probability at
least $1-\delta$, we have
\begin{equation}
    \label{eq:deviation_bound}
    \norm{ X }_\X \leq 
    \sqrt{ \tr(R) } + \sqrt{2 \log( 1/\delta) \norm{R} }
    \quad
    \text{ for all } \delta > 0.
\end{equation}
A further simplification of \eqref{eq:deviation_bound} gives
\begin{equation}
    \label{eq:hoeffding}
    \P[ \norm{X}_\X > \epsilon ] 
    \leq 
    \exp\left( - \frac{\epsilon^2}{8 \norm{R}} \right)
    \quad 
    \text{for all }
    \epsilon > 2 \sqrt{ \tr(R)}
\end{equation}
as an estimate for the outer tail of $\norm{X}_\X$.
which we revisit in the context of
tail bounds for sums of random vectors
as a Hoeffding-type inequality in \Cref{sec:sums}.

\begin{remark}[Optimality of \Cref{prop:main}]
    \label{rem:sharpness}
    Generalising a well-known definition for real-valued random variables,
    we may call a random vector $X$ \emph{strictly subgaussian} 
    (or strongly subgaussian) in $\X$, if $X$ is
    $\Cov[X]$-subgaussian 
    (\citealp{BuldyginKozachenko2000}, Section 1.2).
    For a strictly subgaussian random vector,
    the bounds in \Cref{prop:main} naturally describe the concentration
    of the random variable $\norm{X}^2_\X$ about its expectation, as in this case
    we have the identity $\tr(R) = \tr(\Cov[X]) = \E[ \norm{ X }_\X^2]$.
    More generally, if $X$ is $R$-subgaussian, we only
    have $\Cov[X] \preceq R$
    which follows immediately since $\E[\xi^2] \leq \sigma^2$
    for any real-valued $\sigma^2$-subgaussian
    random variable $\xi$. 
    In particular, this implies $\E[\norm{X}^2_\X ] \leq \tr(R)$
    in the general case, allowing for a potential improvement of the
    term $\tr(R)$ in the above bounds.
\end{remark}

%

\subsection{Concentration of positive quadratic forms}

We  now see that subgaussian 
variance proxies of linearly transformed
$R$-subgaussian vectors admit
a simple characterisation which is well-known in the Gaussian case.
We consider a second real separable Hilbert space $\Y$.

\begin{lemma}
    \label{lem:subgaussian_linear}
    Let $A: \X \to \Y$ be a bounded linear operator
    and $X$ be $R$-subgaussian in $\X$.
    Then the transformed random vector
    $AX$ is $ARA^*$-subgaussian in $\Y$.
\end{lemma}

\begin{proof}
    For any $y \in \Y$, we may choose $u:= A^*y \in \X$
    in the inequality \eqref{eq:subgaussian}.
\end{proof}

From \Cref{prop:main}, we directly obtain a tail bound
for the quadratic form $X \mapsto \norm{AX}_\Y^2$.
The result directly falls in line with a variety of 
known concentration bounds for quadratic forms,
which we discuss in \Cref{sec:related_work}.

\begin{corollary}[Concentration of quadratic form]
    \label{cor:main}
    Let $A: \X \to \Y$ be a bounded linear operator
    and let $X$ be $R$-subgaussian in $\X$. Then we have
    \begin{equation}
        \label{eq:cor_main}
        \P\left[ 
        \norm{AX}_{\Y}^2 
        >
        \tr(B) +
        2 \sqrt{t} \norm{B}_{S_2(\Y)} 
        + 2 t \norm{B} 
        \right] 
    \leq e^{-t}, \quad t \geq 0
    \end{equation}
    with the trace class operator 
    $B:= A R A^*: \Y \to \Y$.
\end{corollary}

For completeness, we give a condition that ensures
$R$-subgaussianity of a linearly transformed weakly
subgaussian random vector.
It is proven similarly to \Cref{lem:subgaussian_linear}
by noting that we have
$AA^* \in S_1(\Y)$ for every $A \in S_2(\X, \Y)$.

\begin{lemma}
    Let $A \in S_2(\X,\Y)$
    and let $X$ be $\sigma^2$-weakly subgaussian in $\X$
    for some $\sigma^2 >0$.
    Then $AX$ is $\sigma^2 AA^*$-subgaussian in $\Y$.
\end{lemma}

\section{Sums of subgaussian vectors}
\label{sec:sums}

We investigate the tail behaviour of sums of $R$-subgaussian vectors.
Just like for real-valued subgaussians, the variance proxy
for sums of independent $R$-subgaussian random vectors
is obtained as the sum of the individual variance proxies.

\begin{lemma}[Independent $R$-subgaussian sums]
    \label{lem:subgaussian_independent_sum}
    Let $X_1, \dots X_n$ be independent centered random vectors in the
    separable Hilbert space $\X$ such that $X_i$ is $R_i$-subgaussian.
    Then $\sum_{i = 1}^n X_i$ is $\sum_{i=1}^n R_i$-subgaussian.
\end{lemma}

\begin{proof}
Let $S_n := \sum_{i = 1}^n X_i$. By independence and $R_i$-subgaussianity, we 
have
\begin{equation*}
    \log \E[ e^{ \innerprod{u}{S_n}_\X } ]
    = 
    \log \prod_{i=1}^n \E[ e^{ \innerprod{u}{X_i}_\X } ]
    \leq
    \sum_{i=1}^n \frac{ \innerprod{ u }{ R_i u }_\X }{2}, \quad u \in \X.
    \qedhere
\end{equation*}
\end{proof}

Together with the previous results, one obtains concentration bounds for sums
of $R$-subgaussian vectors.

\begin{example}[Hoeffding inequality]
    Let $X_1, \dots X_n$ be independent copies of some $R$-subgaussian
    random vector $X$ in the Hilbert space $\X$.
    According to \Cref{lem:subgaussian_linear} and
    and \Cref{lem:subgaussian_independent_sum}, the normalised sum
    $\sum_{i=1}^n X_i$ is $\tfrac{1}{n}R$-subgaussian.
    We can now apply the bounds obtained 
    in \Cref{sec:concentration}.
    With probability at least $1-\delta$, we have
    \begin{equation*}
        \left\Vert \frac{1}{n} \sum_{i=1}^n X_i \right\Vert_\X 
        \leq
        \frac{ \sqrt{\tr(R)} + \sqrt{ 2 \log(1/\delta) \norm{R} }}{\sqrt{n}}
        \quad \text{for all }
        \delta > 0.
    \end{equation*}
    In particular, this gives the tail estimate 
    \begin{equation*}
        \P\left[ \left\Vert \frac{1}{n} \sum_{i=1}^n  X_i \right\Vert_\X 
        \! > \epsilon \right] 
        \leq 
        \exp\left( - \frac{n \epsilon^2}{8 \norm{R}} \right)
        \quad \text{for all }
        \epsilon > 2 \sqrt{\tfrac{\tr(R)}{n}}.
    \end{equation*}
\end{example}

\section{Quadratic forms with random operators}
\label{sec:random}

We show that \Cref{prop:main} can be used straightforwardly 
to obtain bounds for the 
quadratic form induced by a \emph{random operator}
$A$, i.e., a random variable
$\omega \mapsto A(\omega) \in L(\X, \Y)$ for $\omega \in \Omega$.
As the space $L(\X, \Y)$ is nonseparable when $\X$ and $\Y$ are
infinite-dimensional, we first fix some terminology
in order to avoid technical issues concerning the measurability of $A$.
We refer the reader to \citet{BharuchaReid1972}
for the background on notions of measurability of operator-valued functions
and the theory of random operators.

We say that the operator-valued function $A$ given by $\omega \mapsto A(\omega)$
for $\omega \in \Omega$ is a \emph{random operator} in $L(\X, \Y)$ if
$A(\omega) \in L(\X, \Y)$ for all $\omega \in \Omega$ and
$Au = A(\omega)u$ is a $\Y$-valued random vector in the Bochner sense
for every $u \in \X$.
If $A$ is a random operator in $L(\X, \Y)$ and
$X$ is a random vector in $\X$, then 
$AX$ is a random vector in $\Y$ in the Bochner sense
\citep[Proposition 13]{Dinculeanu2000}.
Moreover, $\norm{A}$ is a real-valued random variable.
We say $A$ is \emph{integrable} if
$\norm{A} \in L^1(\P)$, which we write as $A \in L^1(\P; L(\X, \Y))$.
In this case, we may consider the \emph{expectation} of $A$ as
the unique operator $\E[A] \in L(\X, \Y)$ given by
\[ 
    \E[A]u := \E[Au], \quad \text{for all }u \in \X .
\]
We say that the random operator $A$ is \emph{independent} of some 
random variable $X$
if the $\sigma$-field generated by $\{Au \mid u \in \X\}$ is independent
of $X$.
We consider random operators which are almost surely
bounded in the Loewner sense.

\begin{lemma}
   \label{lem:bounded_selfadjoint_operator}
   Let $A$ be a random operator in $L(\X, \Y)$.
   Assume that there exists some fixed $C \in L(\X, \Y)$ such that
   $A^*A \preceq C^*C$ almost surely.
   Let furthermore $X$ be an integrable centered random vector in $\X$
   such that there exists a positive self-adjoint $R \in S_1(\X)$
   almost surely satisfying 
   \begin{equation}
   \label{eq:operator_condition}
   \log \E[ e^{\innerprod{u}{X}_\X} | A] \leq \frac{\innerprod{u}{Ru}_\X}{2},
   \quad u \in X.
   \end{equation}
   Then $AX$ is $CRC^*$-subgaussian in $\Y$.
\end{lemma}

\begin{proof}
    For all $y \in \Y$ we have
    \begin{align*}
        \E[ e^{ \innerprod{ y }{ AX}_\Y }]
        =
        \E [ \E[ e^{ \innerprod{ y }{ AX}_\Y } | A] ]
        \leq \E[ e^{ \innerprod{ y }{ARA^*y}_\Y /2 }  ].
    \end{align*}
    The claim follows since
    the almost sure condition $A^*A \preceq C^*C$ implies that
    \begin{align*}
        \innerprod{ y }{ ARA^* y}_\Y  
        \leq
        \innerprod{ y}{ C R C^* y}_\Y
    \end{align*}
    almost surely for all $y \in \Y$,
    which we prove in \Cref{lem:loewner3} in the appendix.
\end{proof}

\begin{remark}[Independence]
    Condition \eqref{eq:operator_condition}
    is satisfied if $X$ is $R$-subgaussian and
    $A$ and $X$ are independent.
\end{remark}

\begin{remark}[Self-adjoint operator]
    If $A$ is a random operator in $L(\X)$
    such that $A$ is almost surely self-adjoint and
    $\norm{A} \leq c$ almost surely for some $c \geq 0$,
    then the condition $A^*A \preceq C^*C$ of 
    \Cref{lem:bounded_selfadjoint_operator}
    can be verified with $C := c\idop_\X$.
\end{remark}

We emphasise that
\Cref{lem:bounded_selfadjoint_operator} gives a direct deviation bound for the 
random bilinear form $\norm{AX}^2_\Y$ when combined with \Cref{prop:main}.


\section{Related work}
\label{sec:related_work}

Tail bounds for quadratic forms of subgaussian random variables
can be found in the literature under a variety of assumptions.
A class of results are known as versions of
the \emph{Hanson-Wright inequality}
(for a recent discussion
see e.g.\ \citealp{KlochkovZhivotovskiy2020} and the references therein),
which classically apply to finite-dimensional 
subgaussian vectors with independent components
based on their weak variance proxy. \citet{AdamczakEtAl2020} investigate
quadratic forms involving a finite number of
constant vectors in a Banach space 
with real-valued subgaussian coefficients.
A recent paper by \citet{ChenYang2021} focuses on a similar scenario
in the Hilbert space case with a trace class operator
as variance proxy---this work
is very similar to the considerations presented in our notes.
The authors investigate the deviation of 
a (not necessarily positive) 
quadratic form of a finite number of $R$-subgaussian random vectors.
In particular, for the special case
of the squared norm, \citet{ChenYang2021} obtain the estimate
\begin{equation*}
    \P\left[ \norm{ X }^2_\X - \tr(R) > \epsilon \right] \leq
    \exp \left( -
            \frac{\epsilon^2}%
            { 8 \max \{ \norm{R}^2_{S_2(\X)}, \epsilon\norm{R} \}   }
    \right),
\end{equation*}
which contains a generally suboptimal exponent.
Such a bound arises
from a \emph{subexponential} control of the cumulant-generating function
of $\norm{X}^2_\X$
\citep[e.g.][Section 2.8]{Vershynin2018},
while our bound on the cumulant-generating function
\eqref{eq:main_cgf} is of \emph{sub-gamma} type and follows from
a quite elementary proof, giving the estimate
\begin{equation*}
    \P\left[ \norm{ X }^2_\X - \tr(R) > \epsilon \right] \leq 
    \exp \left( -
        \frac{\epsilon^2}{ 4( \norm{R}^2_{S_2(\X)} + \epsilon \norm{R} ) }
        \right).
\end{equation*}
This type of Bernstein bound
is commonly obtained in the literature under the well-known
\emph{Bernstein moment condition}
\citep[Section 2.8]{BLM2013}.
The difference between those two types of bounds
is particularly relevant in cases 
where we have $\norm{R}^2_{S_2(\X)} \gg \norm{R}$.

The bounds in \Cref{prop:main} and \Cref{cor:main}
can be interpreted as a direct generalisation
of a bound for positive quadratic forms
provided by \citet{Hsu2012}, who show that  
\begin{equation}
    \label{eq:hsu}
        \P\left[ 
        \norm{AX}_{\Y}^2 
        > \sigma^2 \left(
        \tr(A^*A) +
        2 \sqrt{t} \norm{A^*A}_{S_2(\Y)} 
        + 2 t \norm{A^*A} \right)
        \right] 
    \leq e^{-t}
\end{equation} for all $t \geq 0$
whenever $\X$ and $\Y$ are finite-dimensional and
$X$ is $\sigma^2$-weakly subgaussian.
For $R$-subgaussian $\X$,
with $\sigma^2= \norm{R}$ and
$B = A R A^*$, we have
$\sigma^2 \norm{A^*A} = \sigma^2\norm{A}^2 \geq \norm{B} $.
In settings where the eigenvalues of $R$ decay fast,
we generally have
\[
    \sigma^2 \tr(A^*A) \gg \tr(B) \quad \text{as well as} \quad
\sigma^2 \norm{A^*A}_{S_2(\X)} \gg \norm{B}_{S_2(\Y)}.
\]
Therefore, \Cref{prop:main} may also give a tighter bound
in comparison to \eqref{eq:hsu} when the sharper
subgaussian variance proxy \eqref{eq:subgaussian} is available
in finite dimensions. However, with the above choice of $\sigma^2$, 
\Cref{prop:main} implies \eqref{eq:hsu}.
When $\X$ and $\Y$ are infinite-dimensional,
a direct analogue of \eqref{eq:hsu}
does not exist in this general setting; the operator $A^*A$ is generally
not trace class (unless $A$ is assumed to be Hilbert--Schmidt). 

\section{Statistical inverse problem}
\label{sec:inverse}

Our results can be readily applied to a typical setting of
regularised estimators in statistical inverse problems
(see e.g.\ \citealp{BissantzHohageMunk2007}).
We consider the inverse problem associated with the model given by
\begin{equation}
    \label{eq:inverse_problem}
    Y = Tu + \epsilon
\end{equation}
for some $u \in \X$, with the known positive self-adjoint forward operator
$T \in L(\X)$
the centred $\X$-valued noise variable $\epsilon$
which we assume to be $R$-subgaussian.
Our goal is to recover $u$ from the noisy observation
$Y$ via a \emph{spectral regularisation strategy}.
We will not discuss the details of regularisation theory here
and refer the reader to the standard literature
(see e.g.\ \citealp{EnglHankeNeubauer1996}).

We solve the inverse problem by
constructing a regularised estimator of $u$ as
\begin{equation*}
    \hat{u}_\alpha := g_\alpha(T) Y.
\end{equation*}
Here, the regularisation strategy $g_\alpha: [0, \infty) \to \R$
for a \emph{regularisation parameter} $\alpha >0$ is applied
to the operator $T$ via the spectral calculus.
The regularisation strategy $g_\alpha$ is constructed such that
$g_\alpha(T) Tu \to u$ as $\alpha \to 0$,
i.e., $g_\alpha(T)$ approximates the (generally unbounded) inverse of $T$
in a pointwise fashion for reasonable $u$ in its domain.
For specific choices of $g_\alpha$, the estimate
$\hat{u}_\alpha$ may yield a
\emph{ridge regressor} (Tikhonov--Phillips regularisation),
\emph{principal component regressor} (spectral truncation)
or \emph{gradient descent} scheme (Landweber iteration)
when \eqref{eq:inverse_problem} is interpreted
in the context of fixed design regression 
(\citealp{EnglHankeNeubauer1996}, Section 4).

The performance of the estimator $\hat{u}_\alpha$ may be measured
on a continuous scale of errors
via the parametrised term
$\norm{ T^s(\hat{u}_\alpha - u) }_\X $ for $s \in [0, 1]$,
where $s = 0$ corresponds to the classical 
\emph{reconstruction error} and
$s=1$ corresponds to the \emph{prediction error}, which is sometimes
also called the \emph{weak reconstruction error}.
We set
$
u_\alpha := \E[ \hat{u}_\alpha]
= g_\alpha(T) Tu$ and
obtain the \emph{bias-variance decomposition}
\[
    \norm{ T^s(\hat{u}_\alpha - u) }_\X 
    \leq
    \norm{ T^s(\hat{u}_\alpha - u_\alpha) }_\X 
    +
    \norm{ T^s(u_\alpha - u) }_\X .
\]
The behaviour of the deterministic second term on
the right hand side, the bias,
is covered for $\alpha \to 0$
by classical regularisation theory under smoothness assumptions
for the true solution $u$.
Our previous results allow to bound the first term on the right-hand side,
the variance, with high probability.
In fact, we see that we have 
$
\norm{ T^s(\hat{u}_\alpha - u_\alpha) }_\X 
=
\norm{ T^s g_\alpha(T) \epsilon }_\X 
$
and hence for all $\delta >0$, we obtain the bound
\begin{equation}
    \label{eq:var_bound1}
    \norm{ T^s g_\alpha(T) \epsilon }^2_\X 
    \leq \tr( B ) + 
    2 \sqrt{ \log(1/\delta) } \norm{ B }_{S_2(\X)}
    + 
    2 \log(1 /\delta ) \norm{ B }
\end{equation}
with probability at least $1 - \delta$ 
due to \Cref{cor:main} with the trace class operator
$B := T^s g_\alpha(T) R g_\alpha(T) T^s$.
The variance bound \eqref{eq:var_bound1}
combines the interplay of the forward operator $T$, the
regularisation strategy $g_\alpha$ and the noise variance proxy operator
$R$ into one expression via the operator $B$
and flexibly allows for further investigations depending
on more detailed assumptions.

\begin{example}[Noise level and regularisation schedule]
    We consider a variable noise scale in terms
    $R := \frac{\sigma^2}{n} \tilde{R}$
    for some fixed positive self-adjoint $\tilde{R} \in S_1(\X)$ with
    $\sigma^2 > 0$, $n \in \N$ and
    $\norm{\tilde{R}} = 1$, where $n$ is interpreted as a sample size.
    We consider the strong reconstruction error given by $s=0$ 
    and make use of the fact that $g_\alpha$ 
    typically satisfies $\norm{g_\alpha(T)} \leq b \alpha^{-1}$ for some
    constant $b>0$ \citep[Section 4]{EnglHankeNeubauer1996}.
    Then $g_\alpha(T) \epsilon$ is $ (\sigma b)^2 \alpha^{-2} n^{-1} \tilde{R}$-subgaussian
    and we get
    \begin{equation*}
    \norm{ g_\alpha(T) \epsilon }^2_\X 
    \leq 
    \frac{ \sigma^2 b^2}{\alpha^{2} n}  
    \left(
    \tr( \tilde{R} ) + 
    2 \sqrt{ \log(1/\delta) } \norm{ \tilde{R} }_{S_2(\X)}
    +  2 \log(1 /\delta )
    \right)
    \end{equation*}
    with probability at least $1 - \delta$. Note this 
    reflects that $\alpha = \alpha(n)$
    must classically satisfy $\alpha(n)^2 n \to \infty$ as $n\to \infty$ 
    in order to yield a consistent estimator overall.
\end{example}

\section*{Acknowledgements}

This work is supported by the Deutsche Forschungsgemeinschaft (DFG) through 
grant EXC 2046 \textit{MATH+}, Project EF1-19: 
\textit{Machine Learning Enhanced Filtering Methods for Inverse Problems}.
The authors wish to thank Vladimir Spokoiny for dicussions
of the finite-dimensional case.

\bibliography{references}

\appendix

\section{Proof of \Cref{prop:main}}
\label{sec:proof}


We first assume that the space $\X$ is finite-dimensional,
which allows us to use a standard Gaussian majorisation.
When the dimension of $\X$ is finite, there exists a
an isotropic $\X$-valued Gaussian vector $\xi \sim \mathcal{N}(0, \idop_\X)$.
We may assume $\xi$ is independent of $X$.
Then since $\log \E[ e^{\innerprod{u}{\xi}_\X} ] = \norm{u}_\X^2/2$ 
for all $u \in \X$, we have
\begin{equation}
    \label{eq:mgf_bound1}
    \begin{split}
        \E[ e^{ \lambda \norm{X}_\X^2 } ] 
        = 
        \E_\xi[ \E_X[ e^{ \sqrt{ 2\lambda} \innerprod{\xi}{X}_\X } ] ]
        \leq
        \E[ e^{\lambda \innerprod{\xi}{ R \xi }_\X }  ]
    \end{split}
\end{equation}
for all $\lambda \geq 0$, where we apply Fubini's theorem.


As the quadratic form 
$\innerprod{\xi}{ R \xi }_\X = \norm{ R^{1/2} \xi}^2_\X$
is invariant under unitary transformations of $R^{1/2} \xi$, we can apply 
the typical diagonalisation argument outlined for example
by \citet[Example 2.12]{BLM2013}
to the right-hand side of \eqref{eq:mgf_bound1}. 
In particular, \Cref{lem:laurent} applied
to the random variable $R^{1/2}\xi \sim \mathcal{N}(0, R)$
immediately gives
\begin{equation}
    \label{eq:mgf_bound2}
    \log \E[ e^{\lambda \innerprod{\xi}{R\xi}_\X } ]
    \leq 
    \lambda \tr( R ) + \frac{ \lambda^2\norm{R}_{S_2(\X)}^2 }{ 1 - 2 \lambda \norm{ R }},
    \quad
    0 \leq \lambda < 1/2\norm{R}.
\end{equation}


We now assume that $\X$ is infinite-dimensional and
repeat a monotone convergence 
argument by \citet[Lemma 2.4.1]{Yurinsky1995}.
We consider an orthonormal basis $(e_i)_{i \in \N}$ of $\X$
and for $d \in \N$ let $\Pi_d: \X \to \X$ denote the
orthogonal projector onto the $d$-dimensional 
subspace $\X_d := \mspan\{e_1, \dots e_d \} \subset \X$.

We fix the finite-rank operator $R_d :=  \Pi_d R$.
It is easy to show that we have the monotone
convergence
$\norm{ \Pi_d X }^2_\X \to \norm{ X }^2_{\X}$
as $d \to \infty$ by Parseval's identity.
Similarly, we have the monotone convergence
$\tr(R_d) = \sum_{i=1}^d \gamma_i \to \tr(R) $
and
$\norm{ R_d }^2_{S_2(\X)} = \sum_{i=1}^d \gamma_i^2 \to \norm{R}_{S_2(\X)}^2$
due to the invariance of the trace and Hilbert--Schmidt norm
under the choice of orthonormal basis $(e_i)_{i \in \N}$
by considering the eigenvectors of $R$.
Here, $(\gamma_i)_{i \in \N} \in \ell^1(\N)$ denotes the sequence
of nonnegative eigenvalues of $R$
(see e.g.\ \citealp{Weidmann1980}, Section 7.1).
Finally, we have the monotone convergence
$\norm{ R_d } \to \norm{ R}$ as shown by
\citet[Lemma 4.3.8] {Hackbusch1995},
where the monotonicity follows from
expanding the definition of the 
operator norm in terms of Parseval's identity.

We now note that $\Pi_d X$ is $R_d$-subgaussian.
Applying first the monotone convergence theorem
and then the finite-dimensional 
bound shown in \eqref{eq:mgf_bound2}, we obtain
\begin{align}
    \log \E[ e^{\lambda \norm{X}^2_\X } ]
    &= \lim_{d \to \infty}
    \log \E[ e^{\lambda \norm{\Pi_d X}^2_\X } ] \nonumber \\
    &\leq 
    \lim_{d \to \infty}
    \lambda \tr( R_d ) + \frac{ \lambda^2\norm{R_d}_{S_2(\X)}^2 }{ 1 - 2 \lambda \norm{ R_d }}
    \nonumber \\
    &=
    \lambda \tr( R ) + \frac{ \lambda^2\norm{R}_{S_2(\X)}^2 }{ 1 - 2 \lambda \norm{ R }},
    \quad
    0 \leq \lambda < 1/2\norm{R}, \label{eq:mgf_bound3}
\end{align}
which constitutes the infinite-dimensional version of \eqref{eq:mgf_bound2}.

The final probability bound is obtained from the
Chernoff bound for the random variable $\norm{X}^2_\X - \tr(R)$ based on
\eqref{eq:mgf_bound3} as shown by \citet[Section 2.4]{BLM2013}.
\hfill\qedsymbol

\subsection{Loewner partial order and trace}

We collect some general properties of linear operators
used throughout the main text.
The first property generaliseses the cyclic invariance of the trace.
It is standard for two 
Hilbert--Schmidt operators acting on a single Hilbert space,
but not for a trace class operator and a bounded operator
acting between two distinct spaces;
hence we include it for completeness. 

Let $\X$ and $\Y$ be separable Hilbert spaces.

\begin{lemma}
    \label{lem:cyclic}
    Let $A \in S_1(\X, \Y)$ and $B \in L(\Y, \X)$.
    We have
    \[
        \tr( B A ) = \tr( A B ).
    \]
\end{lemma}

A proof based on the singular value decomposition of $A$ is provided by
\citet[Theorem 3.1]{Simon2005}
for the case that $\X = \Y$ but also works 
in the general setting presented here.

We show a trace inequality induced by the Loewner partial order.

\begin{lemma}
    \label{lem:loewner1}
    Let $A \in L(\X)$ self-adjoint and $C, R \in S_1(\X)$
    self-adjoint.
    If $0 \preceq A$ and $C \preceq R$, then we have
    \[
        \tr(AC) \leq \tr( AR )
    \]
    The same conclusion holds under the assumption
    $A \in S_1(\X)$ and $C, R \in L(\X)$.
\end{lemma}

\begin{proof}
    Let $(e_i)_{i \in I}$ an orthonormal basis of $\X$ such that
    $A e_i = \lambda_i e_i$ for $\lambda_i \in [0, \infty)$ by
    the spectral theorem for bounded self-adjoint operators.
    We have
    \begin{align*}
        \tr(AC) 
        &= \sum_{i \in I} \innerprod{ e_i}{ AC e_i}_\X 
        = \sum_{i \in I} \lambda_i \innerprod{ e_i}{ C e_i} \\
        &\leq \sum_{i \in I} \lambda_i \innerprod{ e_i}{ R e_i}
        = \sum_{i \in I} \innerprod{ e_i}{ AR e_i}_\X
        = \tr(AR).
    \end{align*}
\end{proof}

The following fact is standard.

\begin{lemma}
    \label{lem:loewner2}
    Let $A \in L(\X, \Y)$. For self-adjoint operators
    $C, R \in S_1(\X)$ such that $C \preceq R$, we have
    $
        A C A^* \preceq A R A^*.
    $
\end{lemma}

\begin{proof}
    We see that $ 0 \preceq A(R-C)A^* = A R A^* - A C A^* $.
\end{proof}

\begin{remark}
In particular, the above implies $\tr(A C A^*) \leq \tr(A R A^*)$.
\end{remark}

%

For a self-adjoint operator $A \in L(\X)$, it is sometimes
convenient to use the elementary identity
\begin{equation}
    \label{eq:hs_innerprod}
    \innerprod{u}{Au}_\X
    = 
    \innerprod{ A }{ \Pi_u}_{S_2(\X)}
    =
    \tr( A \Pi_u ) = \tr(\Pi_uA)
\end{equation}
for all $u \in \X$, where $\Pi_u$ denotes the orthogal 
projector onto the one-dimensional subspace of $\X$ spanned
by $u$. 

\begin{lemma}
    \label{lem:loewner3}
    Let $A,C \in L(\X, \Y)$ such that
    $A^* A \preceq C^* C$ and $R \in L(\X)$ self-adjoint such that $0 \preceq R$.
    Then $ A R A^* \preceq C R C^*$.
\end{lemma}

\begin{proof}
    For every $Y \in \Y$, we have
    \begin{align*}
        \innerprod{y}{A R A^* y}_\Y
        &=
        \tr( ARA^* \Pi_y )
        =
        \tr( R^{1/2} A^* \Pi_y A R^{1/2} ) &&& \text{(by \Cref{lem:cyclic})}  \\
        &\leq
        \tr( R^{1/2} C^* \Pi_y C R^{1/2} ) = \tr( C R C^* \Pi_y)
        &&& \text{(by \Cref{lem:loewner2})} \\
        &= \innerprod{y}{CRC^*y}_\Y,
    \end{align*}
    where we use that $A^* A \preceq C^* C$ clearly implies
    $A^* \Pi_y A \preceq C^* \Pi_y C$.
\end{proof}

\end{document}